\documentclass[10pt]{article}
\usepackage[utf8]{inputenc}
\usepackage{amsmath}
\usepackage{amssymb}
\usepackage{amsfonts}
\usepackage{enumitem}
\usepackage{amsthm}
\usepackage{geometry}
\usepackage{parskip}
\usepackage{graphicx}
\title{Optimal control on finite graphs: asymptotic optimal controls and ergodic constant in the case of entropic costs\thanks{The author would like to thank Philippe Bergault (Université Paris 1 Panthéon-Sorbonne), Diogo Gomes (Kaust), and Jean-Michel Lasry (Ceremade and Institut Louis Bachelier) for the discussions he had with them on the subject.}}
\author{Olivier Guéant\thanks{Université Paris 1 Panthéon-Sorbonne. Centre d'Economie de la Sorbonne. 106, Boulevard de l'Hôpital, 75013 Paris.}}
\date{}

\newtheorem{theorem}{Theorem}

\newtheorem{lemma}{Lemma}

\makeatletter
\def\blfootnote{\xdef\@thefnmark{}\@footnotetext}
\makeatother

\begin{document}

\maketitle

\abstract{For optimal control problems on finite graphs in continuous time, the dynamic programming principle leads to value functions characterized by systems of nonlinear ordinary differential equations. In this paper, we consider the case of entropic costs for which the nonlinear differential equations can be transformed into linear ones thanks to a change of variables linked to the classical duality between entropy and exponential. When the graph is connected, we show that the asymptotic optimal control and the ergodic constant can be computed very easily with classical tools of matrix analysis.}

\vspace{3mm}


\setlength{\parindent}{0em}

\section{Introduction}

Continuous-time optimal control theory is often presented in the case of a continuous state space. The case of a discrete state space in which the agent controls the intensity of point processes has drawn less attention from academics beyond the seminal work of Brémaud (see for instance \cite{bremaud1981point}). In the recent paper \cite{gueant2020cocv}, general results have been presented in the case of finite graphs. In particular, results about the long-run behavior of value functions and optimal controls have been obtained that echoed existing (but usually more technical) results in the case of continuous state spaces (see also \cite{sam} for an alternative approach with backward stochastic differential equations). Beyond optimal control, the case of mean-field games on finite graphs has also been addressed in \cite{gueant2015existence}.\\

In the vast set of problems tackled thanks to the tools of optimal control theory, linear-quadratic~(LQ) ones in $\mathbb{R}^d$ are known to be among the easiest to solve because they boil down to Riccati equations. They have a lot of applications and often serve as benchmark cases to test numerical methods designed to tackle more general problems. In the case of optimal control on finite graphs in continuous time, there is no equivalent to LQ problems. The nonnegativity constraint on intensities and the intricate topology and geometry of graphs prevent indeed the emergence of phenomena similar to those obtained in Euclidean spaces.\\

There is nevertheless a family of nontrivial continuous-time optimal control problems on finite graphs for which the solution can be computed in (almost-)closed form and can therefore serve as benchmark. This family corresponds to the use of entropic costs to monitor intensities. The well-known duality between exponential and entropy leads indeed to a Hamilton-Jacobi equation for the value function that can be transformed into a linear system of ordinary differential equations. In particular, the value function and the optimal controls can be computed very easily using matrix exponentiation.\\

The use of entropic costs or entropic penalties is common in many fields. The replacement of maxima and minima by there soft versions in many optimization and machine learning problems is intimately related to the use of entropic regularization. Entropic penalties in optimal transport are very common and key to use the famous Sinkhorn's algorithm (see \cite{nutz}). To manage uncertainty in the value of the parameters in optimization and optimal control problems, entropic penalties are a classical tool. In finance, entropic penalties have been introduced to choose amongst all pricing (martingale) measures the ``closest'' one to a reference one, typically the historical probability measure (see also the link with utility-based pricing in \cite{frittelli, rouge}). Another example with mean field games can be found in \cite{gomes}. Many other instances could be cited. Behind entropic penalties and their virtues almost always lie the convex duality between exponential and entropy.\\

Independently of any entropic regularization consideration, optimal control on graphs with entropic costs naturally appeared in \cite{glft} when addressing the market making problem presented in the famous paper \cite{as} by Avellaneda and Stoikov. Market making and more generally inventory management problems are naturally written in the form of optimal control on graphs (see \cite{cartea2015algorithmic, gueant2016financial} for a general presentation on market making). When trading intensities depend exponentially on prices as in the early market making models \cite{as, glft}, the costs turn out to correspond to entropic costs.\\

In this paper, we propose a general study of optimal control problems on finite graphs with entropic costs. In particular, we show that the asymptotic behavior of the value function and the optimal controls obtained in \cite{gueant2020cocv} using comparison principles and semi-group tools on Hamilton-Jacobi equations can be obtained, in the case of a connected graph and entropic costs, thanks to classical results of matrix analysis. Furthermore, we show that the ergodic constant and the asymptotic optimal controls can be computed very easily by finding the largest real eigenvalue of a matrix and an associated eigenvector.\\ 

In Section 2 we introduce the notations and the family of optimal control problems considered throughout the paper. In Section 3, we solve in (almost-)closed form the Hamilton-Jacobi equations associated with these optimal control problems and derive both the value functions and the optimal controls. In Section~4, in the particular case of a connected graph, we study the asymptotic behavior of the value function and optimal controls and derive spectral characterizations of both the ergodic constant and the asymptotic optimal controls.\\

\section{Optimal control problems on graphs with entropic costs}
\label{notation}
Let $T>0$. Let $\left(\Omega,\left(\mathcal{F}_{t}\right)_{t\in [0,T]},\mathbb{P}\right)$
be a filtered probability space, with $\left(\mathcal{F}_{t}\right)_{t\in [0,T]}$
satisfying the usual conditions. We assume that the stochastic processes introduced in this paper are defined on~$\Omega$ and adapted to the filtration $\left(\mathcal{F}_{t}\right)_{t\in [0,T]}$.\\

We consider a finite directed graph $\mathcal{G}$ with no self-loop (\emph{i.e.} there is no edge connecting a node to itself). The set of nodes is denoted by $\mathcal{N} = \{1, \ldots, N\}$ where $N \ge 2$ is an integer. For each node $i \in \mathcal{N}$, we introduce $\mathcal{V}(i) \subset \mathcal{N} \setminus \{i\}$ the neighborhood of the node $i$, \emph{i.e.} the set of nodes $j$ for which a directed edge exists from $i$ to $j$.\\

The graph is said to be connected whenever $\forall i,j \in \mathcal N, \exists K \ge 2, \exists i_1, \ldots, i_K \in \mathcal N$, such that $i_1 = i$, $i_K = j$, and $\forall k \in \{1, \ldots, K-1\}, i_{k+1} \in \mathcal{V}(i_k)$.\\

We consider an agent evolving on the graph $\mathcal{G}$. This agent can choose in continuous time the values of the transition intensities, \emph{i.e.} instantaneous probabilities of transition. At any time $t \in [0,T]$, transition intensities are described by a collection of feedback control functions ${(\lambda_t(i, \cdot))}_{i \in \mathcal{N}}$ where $\lambda_t(i, \cdot): \mathcal{V}(i)\rightarrow \mathbb{R}_+$. We assume that the controls are in the admissible set $\mathcal{A}^T_0$ where, for $t\in [0,T]$,
$$
    \mathcal{A}^T_t= \{(\lambda_s(i, j))_{s \in [t, T], i \in \mathcal{N}, j \in \mathcal{V}(i)} \text{ deterministic,  nonnegative,}$$$$ \text{ such that }\forall i \in \mathcal{N}, \forall j \in \mathcal{V}(i), s \mapsto \lambda_s(i, j) \in L^{1}(t, T)\}.$$

We assume that the instantaneous cost of the agent located at node $i$ is described by a function\footnote{Throughout this paper, the function $x \in \mathbb{R}_+^* \mapsto x \log(x)$ is prolonged by continuity to $x=0$ (by the value $0$).} $$L(i, \cdot): \left(\lambda_{ij}\right)_{j \in \mathcal{V}(i)} \in {\mathbb{R}}^{|\mathcal{V}(i)|}_+ \mapsto L\left(i, \left(\lambda_{ij}\right)_{j \in \mathcal{V}(i)}\right) = - r(i)  + \sum_{j \in \mathcal{V}(i)} \lambda_{ij} \log\left(\frac{\lambda_{ij}}{\bar{\lambda}_{ij}}\right) + \bar\lambda_{ij} - \lambda_{ij}  ,$$ where $|\mathcal{V}(i)|$ is the cardinal of $\mathcal{V}(i)$, $r: \mathcal{N} \rightarrow \mathbb{R}$ a (real) reward function, and $(\bar{\lambda}_{ij})_{i \in \mathcal{N}, j \in \mathcal{V}(i)}$ a family of positive real numbers. For each $i \in \mathcal N$, $r(i)$ can be seen as the instantaneous reward of being in state $i$. The other terms corresponds to the cost paid to monitor moves on the graph. In fact, the transition intensity from node $i$ to node $j$ is $\bar{\lambda}_{ij}$ for each $i \in \mathcal{N}$ and each $j \in \mathcal{V}(i)$ if the agent pays no cost. That rate can be changed upon the payment of a positive cost by the agent and the form of the cost is called entropic for its similarity with entropy (although the sum of intensities is unconstrained).\\

At time $T$, we consider a terminal reward for the agent. This reward depends on her position in the graph and is modelled by a function $g: \mathcal{N} \rightarrow \mathbb{R}$.\\

Let us denote by $(X_s^{t, i, \lambda})_{s \in [t, T]}$ the continuous-time Markov chain on $\mathcal{G}$ starting from node $i$ at time~$t$, with transition intensities given by $\lambda \in \mathcal{A}^T_t$.\\

Starting from a given node $i$ at time $0$, the optimal control problem we consider is the following:\\
\begin{align}
\label{controlpbm}
    & \sup_{\lambda \in \mathcal{A}^T_0} \mathbb{E}\left[ -\int_0^T L\left(X^{0,i,\lambda}_t, \left(\lambda_t\left(X^{0,i,\lambda}_t, j\right)\right)_{j \in \mathcal{V}\left(X^{0,i,\lambda}_t\right)}\right)dt + g\left(X^{0,i,\lambda}_T\right)\right].
\end{align}

For each node $i \in \mathcal{N}$, the value function of the agent, at state $i$, is defined as
\begin{align*}
    & u^T_i(t) = \sup_{\lambda \in \mathcal{A}^T_t} \mathbb{E}\left[ -\int_t^T L\left(X_s^{t, i, \lambda}, \left(\lambda_s\left(X_s^{t, i, \lambda}, j\right)\right)_{j \in \mathcal{V}\left(X^{t,i,\lambda}_s\right)}\right)ds + g\left(X_T^{t, i, \lambda}\right)\right].
\end{align*}

Our goal in the next Section is to compute that value function in (almost-)closed form and to deduce the optimal controls.\\

\section{Solution of the associated Hamilton-Jacobi equation and derivation of the optimal controls}

The Hamilton-Jacobi equation associated with the above optimal control problem is
\begin{equation}
\label{HJ_1}
    \forall i \in \mathcal{N},\quad \frac{d}{dt}{V^T_i}(t) = - \sup_{\left(\lambda_{ij}\right)_{j \in \mathcal{V}(i)} \in \mathbb{R}^{|\mathcal{V}(i)|}_+} \left( \left(\sum_{j \in \mathcal{V}(i)}\lambda_{ij}\left(V^T_j(t) - V^T_i(t)\right) \right) - L\left(i, \left(\lambda_{ij}\right)_{j \in \mathcal{V}(i)}\right)\right),
\end{equation}
with terminal condition
\begin{align*}
    \forall i \in \mathcal{N}, \quad V^T_i(T) = g(i).
\end{align*}

Let us define for all $i \in \mathcal{N}$ the Hamiltonian function associated with the cost function $L(i, \cdot)$:
\begin{eqnarray*}
H(i,\cdot):& \mathbb{R}^{|\mathcal{V}(i)|} & \to \mathbb{R}\\
 &(p_{ij})_{j \in \mathcal{V}(i)}  &\mapsto \sup_{\left(\lambda_{ij}\right)_{j \in \mathcal{V}(i)} \in {\mathbb{R}}^{|\mathcal{V}(i)|}_+}  \left(\left(\sum_{j \in \mathcal{V}(i)} \lambda_{ij} p_{ij}\right) - L\left(i, \left(\lambda_{ij}\right)_{j \in \mathcal{V}(i)}\right)\right).
\end{eqnarray*}

Given the specific entropic form chosen for the cost functions, it is straightforward to compute the Hamiltonian functions in closed form. This is the purpose of the following lemma.

\begin{lemma}
\label{Hform}
$\forall i \in \mathcal{N}, \forall p = (p_{ij})_{j \in \mathcal{V}(i)} \in \mathbb{R}^{|\mathcal{V}(i)|}$,
$$H(i, p) = r(i) + \sum_{j \in \mathcal{V}(i)} \bar\lambda_{ij} (e^{p_{ij}}-1).$$
Moreover, the supremum in the definition of $H(i, p)$ is a maximum, reached when
$$\forall j \in \mathcal{V}(i), \quad \lambda_{ij} = \bar\lambda_{ij} e^{p_{ij}}.$$
\end{lemma}

\begin{proof}
Let us consider $i \in \mathcal{N}$ and $p = (p_{ij})_{j \in \mathcal{V}(i)} \in \mathbb{R}^{|\mathcal{V}(i)|}$.\\

The function $$\left(\lambda_{ij}\right)_{j \in \mathcal{V}(i)} \in {\mathbb{R}}^{|\mathcal{V}(i)|}_+ \mapsto  \left(\sum_{j \in \mathcal{V}(i)} \lambda_{ij} p_{ij}\right) - L\left(i, \left(\lambda_{ij}\right)_{j \in \mathcal{V}(i)}\right)$$ is concave. Its gradient vanishes whenever
$$\forall j \in \mathcal{V}(i),\quad  p_{ij} - \log\left(\frac{\lambda_{ij}}{\bar \lambda_{ij}}\right) - 1 + 1 = 0.$$
Therefore, the supremum in the definition of $H(i, p)$ is in fact a maximum, reached when
$$\forall j \in \mathcal{V}(i), \quad \lambda_{ij} = \bar \lambda_{ij} e^{p_{ij}},$$
and we obtain
\begin{eqnarray*}
H(i, p) &=& r(i) + \sum_{j \in \mathcal{V}(i)} \bar \lambda_{ij} (e^{p_{ij}}-1).
\end{eqnarray*}

\end{proof}

Using the above lemma, we see that the Hamilton-Jacobi equation \eqref{HJ_1} associated with Problem \eqref{controlpbm} writes
\begin{equation}
\label{HJ}
    \forall i \in \mathcal{N},\quad \frac{d}{dt}{V^T_i}(t) = - r(i) - \sum_{j \in \mathcal{V}(i)} \bar \lambda_{ij}\left( \exp\left(V^T_j(t) - V^T_i(t)\right) - 1 \right),
\end{equation}
with terminal condition
\begin{align}
\label{terminal condition}
    \forall i \in \mathcal{N}, \quad V^T_i(T) = g(i).
\end{align}

Eq. \eqref{HJ} with terminal condition \eqref{terminal condition} can be solved in (almost-)closed form. This is the purpose of the next theorem.

\begin{theorem}
\label{solution_vf}
Let $R = \text{diag}(r(1), \ldots, r(N))$ be the diagonal matrix of instantaneous rewards. Let $\bar \Lambda$ be the infinitesimal generator of the Markov chain associated with the intensities $\left(\bar \lambda_{ij}\right)_{i \in \mathcal N, j \in \mathcal{V}(i)}$, i.e.
$$\bar \Lambda_{ij} = \left\{
  \begin{array}{ll}
    \bar \lambda_{ij}, & \text{if } j \in \mathcal{V}(i),\\
    - \sum_{k \in \mathcal V (i)}\bar \lambda_{ik} , & \text{if } j = i,\\
    0, & \text{otherwise.}
  \end{array}
\right.
$$
Let us define $B = R + \bar \Lambda$.\\

Let $\frak{g}$ be the column vector $(e^{g(1)}, \ldots, e^{g(N)})'$.\footnote{$'$ designates the transpose operator throughout this paper.}\\

Then, the function $w^T : t \in [0,T] \mapsto w^T(t) = e^{B(T-t)}\frak{g}$ verifies $$\forall i \in \mathcal{N}, \forall t\in [0,T], \quad w_i^T(t) > 0$$ and
$$v^T: t \in [0,T] \mapsto \left(\log(w_i^T(t))\right)_{i \in \mathcal N}$$ defines a solution to Eq. \eqref{HJ} with terminal condition \eqref{terminal condition}.
\end{theorem}

\begin{proof}
Let us consider a constant $\sigma > -\min_{i \in \mathcal{N}} r(i)$. By definition, $B + \sigma I_N$ is a nonnegative matrix and so is $e^{(B + \sigma I_N)(T-t)} - I_N$ for all $t\in [0,T]$. Since $\frak{g}$ has positive coefficients, the vector $e^{\sigma(T-t)} w^T(t) = e^{(B + \sigma I_N)(T-t)}\frak{g} = \frak{g} + \left(e^{(B + \sigma I_N)(T-t)} - I_N\right) \frak{g}$ has positive coefficients for all $t\in [0,T]$. We deduce that 
$$\forall i \in \mathcal{N}, \forall t\in [0,T], \quad w_i^T(t) > 0.$$

From the above positiveness result, $v^T$ is well defined and we have for all $i \in \mathcal{N}$,
\begin{eqnarray*}
\frac{d}{dt}{v^T_i}(t) &=& \frac1{{w^T_i}(t)} \frac{d}{dt}{w^T_i}(t)\\
&=& \frac1{{w^T_i}(t)}\left( - r(i) {w^T_i}(t) - \sum_{j \in \mathcal{V}(i)} \bar \lambda_{ij} ({w^T_j}(t) - {w^T_i}(t)) \right)\\
&=& -r(i) - \sum_{j \in \mathcal{V}(i)} \bar \lambda_{ij} \left(\frac{{w^T_j}(t)}{{w^T_i}(t)} - 1\right)\\
&=& -r(i) - \sum_{j \in \mathcal{V}(i)} \bar \lambda_{ij} \left(\exp\left(v^T_j(t) - v^T_i(t)\right) - 1\right).\\
\end{eqnarray*}
Because for all $i \in \mathcal N$, $v_i^T(T) = \log(e^{g(i)}) = g(i)$, $v^T$ defines a solution to Eq. \eqref{HJ} with terminal condition \eqref{terminal condition}.\\
\end{proof}

By using a standard verification argument, we obtain from Lemma \ref{Hform} and Theorem \ref{solution_vf} the solution to Problem \eqref{controlpbm}:\\

\begin{theorem}
\label{theocontrol}
We have:\\
\begin{itemize}
\item $\forall i \in \mathcal N, \forall t  \in [0,T], u^T_i(t) = v^T_i(t) = \log(w^T_i(t))$.\\
\item The optimal controls for Problem \eqref{controlpbm} are given in feedback form by:
$$\forall i \in \mathcal{N}, \forall j \in \mathcal{V}(i), \forall t \in [0,T], \quad \lambda^{T*}_t(i,j) = \bar \lambda_{ij} \frac{w^T_j(t)}{w^T_i(t)}.$$
\end{itemize}
\end{theorem}

\section{Ergodic constant and asymptotic optimal controls}
\label{lt}

The behavior of value functions when the time horizon $T$ tends to $+\infty$ is a classical topic in the optimal control literature. In $\mathbb{R}^d$, there is indeed an extensive literature on the long-run behavior of solutions of Hamilton-Jacobi equations (see for instance \cite{bs, fathi, namah}). In the case of connected graphs, the long-run behavior of value functions and the existence of asymptotic optimal controls have been studied in \cite{gueant2020cocv} with tools inspired from the literature on Hamilton-Jacobi equations (in particular ideas related to semi-groups). In general, when they exist, the ergodic constant and the asymptotic optimal controls can only be found numerically, using finite difference schemes most of the time. However, given the expression of $w^T$, the asymptotic analysis of Problem \eqref{controlpbm} can be carried out independently of the existing literature, with spectral tools, as soon as $\mathcal G$ is connected. This is the purpose of the following theorem.\\

\begin{theorem} 

If $\mathcal G$ is connected, then the real spectrum $\text{Sp}_{\mathbb R}(B)$ of the matrix $B$ is a nonempty set and $\gamma = \max \text{Sp}_{\mathbb R}(B)$ is an algebraically simple eigenvalue whose associated eigenspace is directed by a column vector with positive coefficients, hereafter denoted by $f$.\\

$\gamma$ is the ergodic constant associated with Problem \eqref{controlpbm} and
$$\exists \alpha \in \mathbb R, \forall i \in \mathcal{N}, \forall t \in \mathbb{R}, \quad  \lim_{T \to + \infty} u_i^T(t) - \gamma(T-t) = \alpha + \log(f_i).$$

The resulting asymptotic behavior of the optimal quotes is given by
$$\forall i \in \mathcal{N}, \forall j \in \mathcal{V}(i), \forall t \in \mathbb{R}, \quad \lim_{T \to + \infty} \lambda^{T*}_t(i,j) = \bar\lambda_{ij} \frac{f_j}{f_i}.$$
\end{theorem}

\begin{proof}
Let us consider a constant $\sigma > -\min_{i \in \mathcal{N}} r(i)$. Then, let us denote by $B(\sigma)$ the nonnegative matrix $B + \sigma I_N$.\\

The matrix $\tilde{B}(\sigma)$ defined by $$\forall (i,j) \in \mathcal{N}^2, \qquad \tilde{B}_{ij}(\sigma) = 1_{B_{ij}(\sigma) \neq 0}$$ is the adjacency matrix of a connected graph (the graph $\mathcal G$ to which self-loops have been added). Therefore, $\tilde{B}(\sigma)$ is an irreducible matrice, and so is $B(\sigma)$.\\

By Perron-Frobenius theorem, the spectral radius $\rho(\sigma)$ of $B(\sigma)$ is an algebraically simple eigenvalue of $B(\sigma)$ and the associated eigenspace is directed by a column vector $f$ with positive coefficients.\\

In particular $\text{Sp}_{\mathbb R}(B)$ is a nonempty set and its maximum $\gamma$, equal to $\rho(\sigma) -  \sigma$, is an algebraically simple eigenvalue of $B$ whose associated eigenspace is also directed by $f$.\\

Similarly, $\rho(\sigma)$ is an algebraically simple eigenvalue of $B(\sigma)'$ and the associated eigenspace is directed by a column vector $\phi$ with positive coefficients.\\

Using a Jordan decomposition of $B(\sigma)$, we see that $\frak{g}$ can be written as $\beta f + h$ where $\beta \in \mathbb{R}$ and $h \in \text{Im}(B(\sigma) - \rho(\sigma) I_N)$.\\

Since $\text{Im}(B(\sigma) - \rho(\sigma) I_N) = \text{Ker}(B(\sigma)' - \rho(\sigma) I_N)^\perp = \text{span}(\phi)^\perp$, we have $\frak g - \beta f \perp \phi$. All the coefficients of $\frak g$, $f$, and  $\phi$ being positive, we must have $\beta > 0$.\\

Now,
$$w^T(t) = e^{B(T-t)}\frak g = e^{-\sigma(T-t)}e^{B(\sigma)(T-t)}\frak g  = \beta e^{(\rho(\sigma) - \sigma)(T-t)} f + e^{-\sigma(T-t)} e^{B(\sigma)(T-t)}h.$$
Therefore,
$$ e^{-\gamma(T-t)} w^T(t) = e^{-(\rho(\sigma) - \sigma)(T-t)} w^T(t) = \beta f + e^{(B(\sigma) - \rho(\sigma) I_N) (T-t)}h \to_{T \to +\infty} \beta f.$$

By taking logarithms, we obtain that
$$\forall i \in \mathcal{N}, \quad  \lim_{T \to + \infty} u_i^T(t) - \gamma(T-t) = \log(\beta) + \log(f_i),$$
hence the result for $\alpha = \log(\beta)$. In particular, $\gamma = \lim_{T \to +\infty} \frac{u_i^T(t)}T$ (independently of $t$), and $\gamma$ is therefore the ergodic constant associated with Problem \eqref{controlpbm}.\\

For optimal controls, we obtain
\begin{eqnarray*}
\forall i \in \mathcal{N}, \forall j \in \mathcal{V}(i), \forall t \in [0,T], \quad \lambda^{T*}_t(i,j) &=& \bar{\lambda}_{ij} \frac{w^T_j(t)}{w^T_i(t)}\\
&=& \bar{\lambda}_{ij} \frac{e^{-\gamma(T-t)}w^T_j(t)}{e^{-\gamma(T-t)}w^T_i(t)} \to_{T \to +\infty} \bar{\lambda}_{ij} \frac{f_j}{f_i},
\end{eqnarray*} hence the result.
\end{proof}

This theorem states in particular that solving the ergodic problem associated with Problem \eqref{controlpbm} when $\mathcal{G}$ is connected simply boils down to finding the largest real eigenvalue -- and an associate eigenvector -- of a matrix that depends on the structure of the graph $\mathcal{G}$ and the parameters defining the reward/cost functions $(L(i, \cdot))_{i \in \mathcal N}$.\\

\bibliographystyle{plain}

\section*{Compliance with Ethical Standards}

\begin{itemize}
  \item Conflict of Interest: The author declares that he has no conflict of interest associated with this work.
  \item Funding: The author declares that there was no funding associated with this work.
\end{itemize}

\end{document}